\documentclass[12pt, reqno]{amsart}
\UseRawInputEncoding
\usepackage{amsmath, amsthm, amscd, amsfonts, amssymb, graphicx, color}
\usepackage[bookmarksnumbered, colorlinks, plainpages]{hyperref}
\hypersetup{colorlinks=true,linkcolor=red, anchorcolor=green,
citecolor=cyan, urlcolor=red, filecolor=magenta, pdftoolbar=true}

\newtheorem{theorem}{Theorem}[section]

\newtheorem{corollary}[theorem]{Corollary}
\theoremstyle{definition}
\newtheorem{definition}[theorem]{Definition}
\newtheorem{example}[theorem]{Example}

\theoremstyle{remark}
\newtheorem{remark}[theorem]{Remark}
\numberwithin{equation}{section}
\usepackage{color}

\begin{document}

\setcounter{page}{1}

\title[Inequalities for the Hilbert--Schmidt norm: A geometric approach]
{A geometric approach to inequalities for the Hilbert--Schmidt norm}

\author[A.~Zamani]
{Ali Zamani}

\address{School of Mathematics and Computer Sciences, Damghan University, P.O.BOX 36715-364, Damghan, Iran}
\email{zamani.ali85@yahoo.com}


\begin{abstract}
We define angle $\Theta_{X,Y}$ between non-zero Hilbert--Schmidt operators $X$ and $Y$ by
$\cos\Theta_{_{X,Y}} = \frac{{\rm Re}{\rm Tr}(Y^*X)}{{\|X\|}_{_2}{\|Y\|}_{_2}}$,
and give some of its essentially properties.
It is shown, among other things, that
\begin{align*}
\big|\cos\Theta_{_{X,Y}}\big|\leq \min\left\{\sqrt{\cos\Theta_{_{|X^*|,|Y^*|}}}, \sqrt{\cos\Theta_{_{|X|,|Y|}}}\right\}.
\end{align*}
It enables us to provide alternative proof of some well-known inequalities for the Hilbert--Schmidt norm.
In particular, we apply this inequality to prove Lee's conjecture
[Linear Algebra Appl. 433 (2010), no.~3, 580--584]
as follows
\begin{align*}
{\big\|X + Y\big\|}_{_2} \leq \sqrt{\frac{\sqrt{2} + 1}{2}}\,{\big\|\,|X| + |Y|\,\big\|}_{_2}.
\end{align*}
A numerical example is presented to show the constant $\sqrt{\frac{\sqrt{2} + 1}{2}}$ is smallest possible.
Other related inequalities for the Hilbert--Schmidt norm are also considered.
\end{abstract}
\maketitle
\section{Introduction and preliminaries}
Let $\mathcal{B}(\mathcal{H})$ denote the $C^*$-algebra of all bounded linear operators acting on
a Hilbert space $\big(\mathcal{H}, [\cdot, \cdot]\big)$.
For every $X\in\mathcal{B}(\mathcal{H})$, let $|X|$ denote the square root of $X^*X$, that is, $|X| = (X^*X)^{1/2}$.
Let $X = U|X|$ be the polar decomposition of $X$, where $U$ is some partial isometry.
The polar decomposition satisfies
\begin{align}\label{I.1.1}
U^*X = |X|, \, U^*U|X| = |X|, \, U^*UX = X, \, X^*=|X|U^*, \, |X^*| = U|X|U^*.
\end{align}
Let $\mathcal{C}_1(\mathcal{H})$ and $\mathcal{C}_2(\mathcal{H})$
denote the trace class and the Hilbert--Schmidt class in $\mathcal{B}(\mathcal{H})$ respectively.
It is well known that $\mathcal{C}_1(\mathcal{H})$ and $\mathcal{C}_2(\mathcal{H})$ each form a
two-sided $*$-ideal in $\mathcal{B}(\mathcal{H})$ and $\mathcal{C}_2(\mathcal{H})$ is itself a Hilbert space with the inner product
\begin{align}\label{I.1.5.1}
\langle X, Y\rangle = \displaystyle{\sum_{i=1}^{\infty} [Xe_i, Ye_i] = {\rm Tr}(Y^*X)}
\end{align}
where $\{e_i\}^{\infty}_{i=1}$ is any orthonormal basis of $\mathcal{H}$ and ${\rm Tr}(\cdot)$ is the natural trace on $\mathcal{C}_1(\mathcal{H})$.
Three principal properties of the trace are that it is a linear functional and, for every $X$ and $Y$, we have
${\rm Tr}(X^*) = \overline{{\rm Tr}(X)}$ and ${\rm Tr}(XY) = {\rm Tr}(YX)$.
The Hilbert--Schmidt norm of $X\in\mathcal{C}_2(\mathcal{H})$ is given by ${\|X\|}_{_2} = \sqrt{\langle X, X\rangle}$.
One more fact that will be needed the sequel is that if $X\in\mathcal{C}_2(\mathcal{H})$, then
\begin{align}\label{I.2.1}
{\|X\|}_{_2} = {\|X^*\|}_{_2} = {\big\|\,|X|\,\big\|}_{_2} = {\big\|\,|X^*|\,\big\|}_{_2}.
\end{align}
The reader is referred to \cite{Mu, Simon} for further properties of the Hilbert--Schmidt class.

In Section 2, we define the angle $\Theta_{X,Y}$ between non-zero Hilbert--Schmidt operators
$X, Y$ and present some of its essentially properties.
We obtain the cosine theorem for Hilbert--Schmidt operators and propose
some geometric results. In particular, we state an extension of the Fuglede--Putnam theorem
modulo the Hilbert--Schmidt class.
We also present a triangle inequality for angles in the Hilbert--Schmidt class.
Moreover, we prove that
$|\langle X, Y\rangle|^2 \leq \langle |X^*|, |Y^*|\rangle\langle |X|, |Y|\rangle$
and then we apply it to obtain
$\cos^2\Theta_{_{X,Y}} \leq \cos\Theta_{_{|X^*|,|Y^*|}}\cos\Theta_{_{|X|,|Y|}}$.

In Section 3, by using the results in Section 2,
we provide alternative proof of some well-known inequalities for the Hilbert--Schmidt norm.
For example we present a considerably briefer proof of an extension of the Araki--Yamagami inequality \cite{Araki.Yamagami}
obtained by F.~Kittaneh \cite{Kit.3}.
In particular, we prove Lee's conjecture \cite[p.~584]{Lee} on the sum of the square roots of operators.
Some related inequalities and numerical examples are also presented.
\section{Angle between two Hilbert--Schmidt operators}
Let $X, Y\in \mathcal{C}_2(\mathcal{H})$.
Since $\mathcal{C}_2(\mathcal{H})$ is a Hilbert space with the inner product \eqref{I.1.5.1},
by the Cauchy--Schwarz inequality, we have
\begin{align}\label{T.I.2.11}
-{\|X\|}_{_2}{\|Y\|}_{_2}\leq -|\langle X, Y\rangle| \leq {\rm Re}\langle X, Y\rangle\leq |\langle X, Y\rangle| \leq {\|X\|}_{_2}{\|Y\|}_{_2}.
\end{align}
Therefore, when $X$ and $Y$ are non-zero operators, the inequality \eqref{T.I.2.11} implies
\begin{align*}
-1\leq \frac{{\rm Re}\langle X, Y\rangle}{{\|X\|}_{_2}{\|Y\|}_{_2}} \leq 1.
\end{align*}
This motivates (see also \cite{B.C.M.W.Z, Z.M.F}) defining the angle between $X$ and $Y$ as follows.
\begin{definition}
For non-zero operators $X, Y \in\mathcal{C}_2(\mathcal{H})$, the angle $\Theta_{X,Y}$ between $X$ and $Y$ is defined by
\begin{align*}
\cos\Theta_{_{X,Y}} = \frac{{\rm Re}\langle X, Y\rangle}{{\|X\|}_{_2}{\|Y\|}_{_2}}; \quad 0\leq \Theta_{X,Y} \leq \pi.
\end{align*}
\end{definition}
\begin{remark}\label{R.12}
Let $X, Y\in \mathcal{C}_2(\mathcal{H})\setminus\{0\}$.
\begin{itemize}
\item[(i)] Since $0\leq \Theta_{_{X,Y}} \leq \pi$, we have $\sin\Theta_{_{X,Y}} = \sqrt{1-\cos^2\Theta_{_{X,Y}}}$.
\item[(ii)] If $X$ and $Y$ are positive, then ${\rm Re}\langle X, Y\rangle\geq 0$ and hence $0\leq \Theta_{_{X,Y}}\leq \frac{\pi}{2}$.
\item[(iii)] One can see that $\cos\Theta_{_{\gamma X,\gamma Y}} = \cos\Theta_{_{X,Y}}$ for all $\gamma \in\mathbb{C}\setminus\{0\}$.
\item[(iv)] For every $\alpha, \beta\in\mathbb{R}\setminus\{0\}$, it is easy to check that
\begin{align*}
\cos\Theta_{_{\alpha X, \beta Y}}=\left\lbrace \begin{array}{ll}
\cos\Theta_{_{X,Y}}  & \alpha\beta>0 \\
-\cos\Theta_{_{X,Y}} & \alpha\beta<0. \end{array}\right.
\end{align*}
\end{itemize}
\end{remark}
\begin{definition}\label{de.31}
Let $X$ and $Y$ be non-zero Hilbert--Schmidt operators on $\mathcal{H}$.
\begin{itemize}
\item[(i)] $X$ is called weak orthogonal to $Y$, denoted by $X\perp_{_{w}}Y$, if $\cos \Theta_{_{X,Y}} = 0$.
\item[(ii)] $X$ is called weak parallel to $Y$, denoted by $X\parallel_{_{w}}Y$, if $\sin \Theta_{_{X,Y}} = 0$.
\end{itemize}
\end{definition}
\begin{example}
Let us recall that by \cite[p.~66]{Mu} we have
\begin{align*}
{\rm Tr}(a \otimes b) = [a, b] \quad \mbox{and} \quad {\|a \otimes b\|}_{_2} = \|a\|\|b\|,
\end{align*}
for all $a, b \in\mathcal{H}$.
Here, $a \otimes b$ denotes the rank one operator in $\mathcal{B}(\mathcal{H})$ defined by
$(a \otimes b)c := [c, b]a$ for all $c \in\mathcal{H}$.
Now, let $x, y, z \in\mathcal{H}\setminus\{0\}$. Put $X=x\otimes z$ and $Y=y\otimes z$.
A simple calculation shows that $\langle X, Y\rangle = \|z\|^2[x, y]$.
Thus,
\begin{align*}
\cos \Theta_{_{X,Y}} = \frac{{\rm Re}[x, y]}{\|x\|\|y\|}.
\end{align*}
In particular, $x\otimes z \perp_{_{w}}y\otimes z$ if and only if $x\perp_{_{w}}y$ (i.e., satisfies ${\rm Re}[x, y] =0$).
\end{example}
First we obtain the cosine theorem for Hilbert--Schmidt operators.
\begin{theorem}\label{T.14}
If $X$ and $Y$ are non-zero Hilbert--Schmidt operators on $\mathcal{H}$, then
\begin{align*}
{\|X \pm Y\|}^2_{_2} = {\|X\|}^2_{_2} + {\|Y\|}^2_{_2} \pm 2{\|X\|}_{_2}{\|Y\|}_{_2}\cos\Theta_{_{X,Y}}.
\end{align*}
\end{theorem}
\begin{proof}
By \eqref{I.1.5.1} and Definition \ref{de.31}, we have
\begin{align*}
{\|X \pm Y\|}^2_{_2} & = \langle X \pm Y, X \pm Y\rangle
\\& = \langle X, X\rangle + \langle Y, Y\rangle \pm \langle X, Y\rangle \pm \langle Y, X\rangle
\\& = {\|X\|}^2_{_2} + {\|Y\|}^2_{_2} \pm 2{\rm Re}\langle X, Y\rangle
= {\|X\|}^2_{_2} + {\|Y\|}^2_{_2} \pm 2{\|X\|}_{_2}{\|Y\|}_{_2}\cos\Theta_{_{X,Y}}.
\end{align*}
\end{proof}
As an immediate consequence of Theorem \ref{T.14}, we get the Pythagorean relation for Hilbert-Schmidt operators.
\begin{corollary}\label{T.15}
If $X$ and $Y$ are non-zero Hilbert--Schmidt operators on $\mathcal{H}$,
then $X\perp_{_{w}}Y$ if and only if ${\|X + Y\|}^2_{_2} = {\|X\|}^2_{_2} + {\|Y\|}^2_{_2}$.
\end{corollary}
The following result is another immediate consequence of Theorem \ref{T.14}.
\begin{corollary}\label{T.151}
If $X$ and $Y$ are non-zero Hilbert--Schmidt operators on $\mathcal{H}$,
then $X\parallel_{_{w}}Y$ if and only if ${\|X + Y\|}_{_2} = \Big|{\|X\|}_{_2} \pm {\|Y\|}_{_2}\Big|$.
\end{corollary}
Let us recall the Fuglede--Putnam theorem modulo the Hilbert--Schmidt class due to G.~Weiss \cite[Theorem~1]{Weiss}:
if $X$ and $Y$ are normal operators and $Z$ is a operator on $\mathcal{H}$, then
${\|XZ - ZY\|}_{_2} = {\|X^*Z - ZY^*\|}_{_2}$.
In the following theorem, we present an extension of the Fuglede--Putnam theorem
modulo the Hilbert--Schmidt class which provided by Kittaneh in \cite{Kit.4}.
Our proof is very different from that in \cite[Theorem~5]{Kit.4}.
\begin{theorem}\label{T.1409}
If $X, Y$ and $Z$ are operators on $\mathcal{H}$, then
\begin{align*}
{\|XZ - ZY\|}^2_{_2} + {\|X^*Z\|}^2_{_2} + {\|ZY^*\|}^2_{_2} = {\|XZ\|}^2_{_2} + {\|ZY\|}^2_{_2} + {\|X^*Z - ZY^*\|}^2_{_2}.
\end{align*}
\end{theorem}
\begin{proof}
We may assume that $XZ, ZY, X^*Z, ZY^*\neq0$ otherwise the desired equality trivially holds.
By \eqref{I.1.5.1} and Definition \ref{de.31} we have
\begin{align*}
{\|XZ\|}_{_2}{\|ZY\|}_{_2}\cos\Theta_{_{XZ,ZY}} & = {\rm Re}\langle XZ, ZY\rangle
= {\rm Re}{\rm Tr}\Big(Y^*Z^*XZ\Big)
\\&= {\rm Re}\overline{{\rm Tr}\Big(Z^*X^*ZY\Big)}
= {\rm Re}{\rm Tr}\Big(Z^*X^*ZY\Big)
\\&= {\rm Re}{\rm Tr}\Big(YZ^*X^*Z\Big)
= {\rm Re}\langle X^*Z, ZY^*\rangle
\\&= {\|X^*Z\|}_{_2}{\|ZY^*\|}_{_2}\cos\Theta_{_{X^*Z,ZY^*}},
\end{align*}
and hence
\begin{align}\label{T.I.1.1489}
{\|XZ\|}_{_2}{\|ZY\|}_{_2}\cos\Theta_{_{XZ,ZY}} = {\|X^*Z\|}_{_2}{\|ZY^*\|}_{_2}\cos\Theta_{_{X^*Z,ZY^*}}.
\end{align}
So, by \eqref{T.I.1.1489} and Theorem \ref{T.14}, we have
\begingroup\makeatletter\def\f@size{11}\check@mathfonts
\begin{align*}
{\|XZ - ZY\|}^2_{_2} &+ {\|X^*Z\|}^2_{_2} + {\|ZY^*\|}^2_{_2}
\\&= {\|XZ\|}^2_{_2} + {\|ZY\|}^2_{_2} - 2{\|XZ\|}_{_2}{\|ZY\|}_{_2}\cos\Theta_{_{XZ,ZY}}
+ {\|X^*Z\|}^2_{_2} + {\|ZY^*\|}^2_{_2}
\\&= {\|XZ\|}^2_{_2} + {\|ZY\|}^2_{_2} - 2{\|X^*Z\|}_{_2}{\|ZY^*\|}_{_2}\cos\Theta_{_{X^*Z,ZY^*}}
+ {\|X^*Z\|}^2_{_2} + {\|ZY^*\|}^2_{_2}
\\& = {\|XZ\|}^2_{_2} + {\|ZY\|}^2_{_2} + {\|X^*Z - ZY^*\|}^2_{_2}.
\end{align*}
\endgroup
\end{proof}
\begin{remark}\label{R.128}
Let $X$ and $Y$ be normal operators and let $Z$ be a operator on $\mathcal{H}$.
By \eqref{I.1.5.1} we have
\begin{align*}
{\|X^*Z\|}^2_{_2} + {\|ZY^*\|}^2_{_2} & = \langle X^*Z, X^*Z\rangle + \langle ZY^*, ZY^*\rangle
\\&= {\rm Tr}\Big(Z^*XX^*Z\Big) + {\rm Tr}\Big(YZ^*ZY^*\Big)
\\& = {\rm Tr}\Big(Z^*XX^*Z\Big) + {\rm Tr}\Big(Z^*ZY^*Y\Big)
\\& = {\rm Tr}\Big(Z^*X^*XZ\Big) + {\rm Tr}\Big(Z^*ZYY^*\Big)
\\& = {\rm Tr}\Big(Z^*X^*XZ\Big) + {\rm Tr}\Big(Y^*Z^*ZY\Big)
\\& = \langle XZ, XZ\rangle + \langle ZY, ZY\rangle = {\|XZ\|}^2_{_2} + {\|ZY\|}^2_{_2}.
\end{align*}
Therefore if in Theorem \ref{T.1409}, $X$ and $Y$ are assumed to be normal operators,
then we retain the Fuglede-Putnam theorem modulo the Hilbert-Schmidt class.
\end{remark}
In the following result we present a triangle inequality for angles in the Hilbert--Schmidt class.
\begin{theorem}\label{T.189}
If $X, Y$ and $Z$ are non-zero Hilbert--Schmidt operators on $\mathcal{H}$, then
\begin{align*}
\sin\Theta_{_{X,Y}} \leq \sin\Theta_{_{X,Z}} + \sin\Theta_{_{Z,Y}}.
\end{align*}
\end{theorem}
\begin{proof}
First note that, since $\mathcal{C}_2(\mathcal{H})$ is a Hilbert space,
by \cite{Rao} we have
\begin{align}\label{T.I.1.18}
\Theta_{_{X,Y}} \leq \Theta_{_{X,Z}} + \Theta_{_{Z,Y}}.
\end{align}
If ${\rm Re}\langle X, Z\rangle\leq 0$, we replace $X$ by $-X$; if
${\rm Re}\langle Z, Y\rangle\leq 0$, we replace $Y$ by $-Y$.
Therefore, we may assume that
$0\leq \Theta_{_{X,Z}}, \Theta_{_{Z,Y}} \leq \frac{\pi}{2}$ and $0\leq \Theta_{_{X,Y}} \leq \pi$.
Now, we consider two cases.\\

\textbf{Case 1.} $0\leq \Theta_{_{X,Z}} + \Theta_{_{Z,Y}} \leq \frac{\pi}{2}$.
Since $\sin \theta$ is a increasing function of $\theta$ in the interval $[0, \frac{\pi}{2}]$, by the inequality \eqref{T.I.1.18}, we have
\begin{align*}
\sin\Theta_{_{X,Y}} \leq \sin\big(\Theta_{_{X,Z}} + \Theta_{_{Z,Y}}\big)
= \sin\Theta_{_{X,Z}}\cos\Theta_{_{X,Z}} + \cos\Theta_{_{Z,Y}}\sin\Theta_{_{Z,Y}},
\end{align*}
and hence $\sin\Theta_{_{X,Y}} \leq \sin\Theta_{_{X,Z}} + \sin\Theta_{_{Z,Y}}$.

\textbf{Case 2.} $\frac{\pi}{2}< \Theta_{_{X,Z}} + \Theta_{_{Z,Y}} \leq \pi$.
Then $0 \leq \frac{\pi}{2} - \Theta_{_{Z,Y}} < \Theta_{_{X,Z}} \leq \frac{\pi}{2}$.
Hence,
\begin{align*}
\sin\Theta_{_{X,Y}} \leq 1 &\leq \sqrt{2}\sin\left(\frac{\pi}{4} + \Theta_{_{Z,Y}}\right)
\\& = \cos\Theta_{_{Z,Y}} + \sin\Theta_{_{Z,Y}}
\\& = \sin(\frac{\pi}{2} - \Theta_{_{Z,Y}}) + \sin\Theta_{_{Z,Y}} \leq
\sin\Theta_{_{X,Z}} + \sin\Theta_{_{Z,Y}},
\end{align*}
and the proof is completed.
\end{proof}
As consequences of Theorem \ref{T.189}, we have the following results.
\begin{corollary}\label{T.21}
Let $X, Y, Z\in \mathcal{C}_2(\mathcal{H})\setminus\{0\}$.
If $X\parallel_{_{w}}Z$ and $Z\parallel_{_{w}}Y$, then $X\parallel_{_{w}}Y$.
\end{corollary}
\begin{proof}
Suppose that $X\parallel_{_{w}}Z$ and $Z\parallel_{_{w}}Y$. Then $\sin\Theta_{_{X,Z}} = \sin\Theta_{_{Z,Y}} = 0$.
Hence, by Theorem \ref{T.189}, we obtain $\sin\Theta_{_{X,Y}}=0$, or equivalently, $X\parallel_{_{w}}Y$.
\end{proof}
\begin{corollary}\label{T.28}
Let $X, Y, Z\in \mathcal{C}_2(\mathcal{H})\setminus\{0\}$.
If $X\perp_{_{w}}Y$ and $Z\parallel_{_{w}}Y$, then $X\perp_{_{w}}Z$.
\end{corollary}
\begin{proof}
Suppose that $X\perp_{_{w}}Y$ and $Z\parallel_{_{w}}Y$. Then $\cos\Theta_{_{X,Y}} = 0$ and $\sin\Theta_{_{Z,Y}} = 0$.
Hence $\sin\Theta_{_{X,Y}} = 1$ and $\sin\Theta_{_{Z,Y}} = 0$.
Now, by Theorem \ref{T.189}, we get $\sin\Theta_{_{X,Z}}=1$, or equivalently, $\cos\Theta_{_{X,Z}}=0$.
Thus $X\perp_{_{w}}Z$.
\end{proof}
We now state a interesting inequality based on some ideas of \cite[Lemma~4.1]{Kennedy.Skoufranis}.
\begin{theorem}\label{T.11}
If $X$ and $Y$ are Hilbert--Schmidt operators on $\mathcal{H}$, then
\begin{align}\label{T.I.1.11}
|\langle X, Y\rangle|^2 \leq \langle |X^*|, |Y^*|\rangle\langle |X|, |Y|\rangle.
\end{align}
Moreover, the inequality in \eqref{T.I.1.11} becomes an equality if and only if
$\zeta Y^*X$ is positive for some scalar $\zeta$.
\end{theorem}
\begin{proof}
Let $X = U|X|$ and $Y = V|Y|$ be the polar
decompositions of $X$ and $Y$, respectively.
By \eqref{I.1.1}, \eqref{I.1.5.1} and the Cauchy–-Schwarz inequality, we have
\begingroup\makeatletter\def\f@size{11}\check@mathfonts
\begin{align*}
|\langle X, Y\rangle|^2 &= \Big|{\rm Tr}\big(|Y|V^*U|X|\big)\Big|^2
\\&= \Big|{\rm Tr}\big(|X|^{1/2}|Y|V^*U|X|^{1/2}\big)\Big|^2
\\& = \Big|\Big\langle |Y|^{1/2}V^*U|X|^{1/2}, |Y|^{1/2}|X|^{1/2}\Big\rangle\Big|^2
\\& \leq {\Big\||Y|^{1/2}V^*U|X|^{1/2}\Big\|}^2_{_2}{\Big\||Y|^{1/2}|X|^{1/2}\Big\|}^2_{_2}
\\& = {\rm Tr}\Big(|X|^{1/2}U^*V|Y|V^*U|X|^{1/2}\Big){\rm Tr}\Big(|X|^{1/2}|Y||X|^{1/2}\Big)
\\& = {\rm Tr}\Big(V|Y|V^*U|X|U^*\Big){\rm Tr}\Big(|Y||X|\Big)
\\& = {\rm Tr}\Big(|Y^*||X^*|\Big){\rm Tr}\Big(|Y||X|\Big)
\\& = {\rm Tr}\Big(|Y^*|^*|X^*|\Big){\rm Tr}\Big(|Y|^*|X|\Big)
= \langle |X^*|, |Y^*|\rangle\langle |X|, |Y|\rangle.
\end{align*}
\endgroup
The inequality in \eqref{T.I.1.11} becomes an equality if and only if
\begin{align*}
|Y|^{1/2}|X|^{1/2} = \zeta|Y|^{1/2}V^*U|X|^{1/2}
\end{align*}
for some $\zeta\in\mathbb{C}$, and hence $|Y|\,|X| = \zeta|Y|V^*U|X|$.
So, by \eqref{I.1.1}, we obtain $|Y|\,|X| = \zeta Y^*X$.
Therefore, the inequality in \eqref{T.I.1.11} becomes an equality if and only if
$\zeta Y^*X$ is positive for some $\zeta\in \mathbb{C}$.
\end{proof}
Here we present one of the main results of this paper.
In fact, the following theorem enables us to provide alternative proof of some well-known inequalities for the Hilbert--Schmidt norm.
In particular, we will use the following theorem to prove Lee's conjecture \cite[p.~584]{Lee}
on the sum of the square roots of operators.
\begin{theorem}\label{T.18}
For non-zero Hilbert--Schmidt operators $X$ and $Y$ on $\mathcal{H}$ the following properties hold.
\begin{itemize}
\item[(i)] $\cos^2\Theta_{_{X,Y}} \leq \cos\Theta_{_{|X^*|,|Y^*|}}\cos\Theta_{_{|X|,|Y|}}$.
\item[(ii)] $\big|\cos\Theta_{_{X,Y}}\big|\leq \min\left\{\sqrt{\cos\Theta_{_{|X^*|,|Y^*|}}}, \sqrt{\cos\Theta_{_{|X|,|Y|}}}\right\}$.
\item[(iii)] $\sin^2\Theta_{_{|X^*|,|Y^*|}} + \sin^2\Theta_{_{|X|,|Y|}} \leq 2\sin^2\Theta_{_{X,Y}}$.
\end{itemize}
\end{theorem}
\begin{proof}
(i) By \eqref{I.2.1}, \eqref{T.I.2.11}, Definition \ref{de.31} and Theorem \ref{T.11} we have
\begin{align*}
\cos^2\Theta_{_{X,Y}} &= \left(\frac{{\rm Re}\langle X, Y\rangle}{{\|X\|}_{_2}{\|Y\|}_{_2}}\right)^2
\\& \leq \frac{|\langle X, Y\rangle|^2}{{\|X\|}^2_{_2}{\|Y\|}^2_{_2}}
\\& \leq \frac{\langle |X^*|, |Y^*|\rangle}{{\|X\|}_{_2}{\|Y\|}_{_2}}
\,\frac{\langle |X|, |Y|\rangle}{{\|X\|}_{_2}{\|Y\|}_{_2}}
\\& = \frac{{\rm Re}\langle |X^*|, |Y^*|\rangle}{{\big\|\,|X^*|\,\big\|}_{_2}{\big\|\,|Y^*|\,\big\|}_{_2}}
\,\frac{{\rm Re}\langle |X|, |Y|\rangle}{{\big\|\,|X|\,\big\|}_{_2}{\big\|\,|Y|\,\big\|}_{_2}}
\\& = \cos\Theta_{_{|X^*|,|Y^*|}}\cos\Theta_{_{|X|,|Y|}}.
\end{align*}

(ii) The proof follows immediately from (i).

(iii) By the arithmetic-geometric mean inequality and (i) we have
\begin{align*}
\sin^2\Theta_{_{|X^*|,|Y^*|}} + \sin^2\Theta_{_{|X|,|Y|}}& = 2-\Big(\cos^2\Theta_{_{|X^*|,|Y^*|}} + \cos^2\Theta_{_{|X|,|Y|}}\Big)
\\& \leq 2-2\cos\Theta_{_{|X^*|,|Y^*|}}\cos\Theta_{_{|X|,|Y|}}
\\& \leq 2 - 2 \cos^2\Theta_{_{X,Y}} = 2\sin^2\Theta_{_{X,Y}}.
\end{align*}
\end{proof}
As consequences of the preceding theorem, we have the following results.
\begin{corollary}\label{T.19}
Let $X, Y\in \mathcal{C}_2(\mathcal{H})\setminus\{0\}$. If $|X^*|\perp_{_{w}}|Y^*|$ or $|X|\perp_{_{w}}|Y|$,
then $X\perp_{_{w}}Y$.
\end{corollary}
\begin{proof}
Let $|X^*|\perp_{_{w}}|Y^*|$ or $|X|\perp_{_{w}}|Y|$. Then $\cos\Theta_{_{|X^*|,|Y^*|}}\cos\Theta_{_{|X|,|Y|}} = 0$.
Hence, by Theorem \ref{T.18} (i), we obtain $\cos\Theta_{_{X,Y}} = 0$, or equivalently, $X\perp_{_{w}}Y$.
\end{proof}
\begin{corollary}\label{T.20}
Let $X, Y\in \mathcal{C}_2(\mathcal{H})\setminus\{0\}$.
If $X\parallel_{_{w}}Y$, then $|X^*|\parallel_{_{w}}|Y^*|$ and $|X|\parallel_{_{w}}|Y|$.
\end{corollary}
\begin{proof}
Suppose that $X\parallel_{_{w}}Y$. Then $\sin\Theta_{_{X,Y}} = 0$.
So, by Theorem \ref{T.18} (iii), we obtain $\sin\Theta_{_{|X^*|,|Y^*|}} = \sin\Theta_{_{|X|,|Y|}} = 0$.
Therefore, $|X^*|\parallel_{_{w}}|Y^*|$ and $|X|\parallel_{_{w}}|Y|$.
\end{proof}
\section{Inequalities for the Hilbert--Schmidt norm}
In this section, by using Theorems \ref{T.14} and \ref{T.18},
we provide alternative proof of some well-known inequalities for the Hilbert--Schmidt norm.
First we present a considerably briefer proof of an extension of the Araki--Yamagami inequality \cite{Araki.Yamagami}
obtained by F.~Kittaneh \cite[Theorem~2]{Kit.3}.
\begin{theorem}\label{T.144}
If $X$ and $Y$ are operators on $\mathcal{H}$, then
\begin{align*}
{\big\|\,|X^*| - |Y^*|\,\big\|}^2_{_2} + {\big\|\,|X| - |Y|\,\big\|}^2_{_2} \leq 2{\|X - Y\|}^2_{_2}.
\end{align*}
\end{theorem}
\begin{proof}
Since the desired inequality trivially holds when $X= 0$ or $Y=0$, we may assume $X, Y\neq0$.
By \eqref{I.2.1}, Theorem \ref{T.14}, Theorem \ref{T.18}(i) and the arithmetic-geometric mean inequality we have
\begingroup\makeatletter\def\f@size{10}\check@mathfonts
\begin{align*}
{\big\|\,|X^*| - |Y^*|\,\big\|}^2_{_2} + {\big\|\,|X| - |Y|\,\big\|}^2_{_2}
& = {\big\|\,|X^*|\,\big\|}^2_{_2} + {\big\|\,|Y^*|\,\big\|}^2_{_2} - 2{\big\|\,|X^*|\,\big\|}_{_2}{\big\|\,|Y^*|\,\big\|}_{_2}\cos\Theta_{_{|X^*|,|Y^*|}}
\\& \qquad + {\big\|\,|X|\,\big\|}^2_{_2} + {\big\|\,|Y|\,\big\|}^2_{_2}
- 2{\big\|\,|X|\,\big\|}_{_2}{\big\|\,|Y|\,\big\|}_{_2}\cos\Theta_{_{|X|,|Y|}}
\\& = 2{\|X\|}^2_{_2} + 2{\|Y\|}^2_{_2} - 2{\|X\|}_{_2}{\|Y\|}_{_2}\Big(\cos\Theta_{_{|X^*|,|Y^*|}} + \cos\Theta_{_{|X|,|Y|}}\Big)
\\& \leq 2{\|X\|}^2_{_2} + 2{\|Y\|}^2_{_2} - 4{\|X\|}_{_2}{\|Y\|}_{_2}\sqrt{\cos\Theta_{_{|X^*|,|Y^*|}}\cos\Theta_{_{|X|,|Y|}}}
\\& \leq 2{\|X\|}^2_{_2} + 2{\|Y\|}^2_{_2} - 4{\|X\|}_{_2}{\|Y\|}_{_2}\big|\cos\Theta_{_{X,Y}}\big|
\\& \leq 2{\|X\|}^2_{_2} + 2{\|Y\|}^2_{_2} - 4{\|X\|}_{_2}{\|Y\|}_{_2}\cos\Theta_{_{X,Y}}
= 2{\|X - Y\|}^2_{_2}.
\end{align*}
\endgroup
\end{proof}
As an immediate consequence of Theorem \ref{T.144},
we get the Araki--Yamagami inequality \cite[Theorem~1]{Araki.Yamagami}.
\begin{corollary}\label{C.147}
If $X$ and $Y$ are operators on $\mathcal{H}$, then
\begin{align*}
{\big\|\,|X| - |Y|\,\big\|}_{_2} \leq \sqrt{2}{\|X - Y\|}_{_2}.
\end{align*}
\end{corollary}
\begin{remark}
In \cite{Araki.Yamagami}, H.~Araki and S.~Yamagami remarked that $\sqrt{2}$ is the best possible
coefficient for a general $X$ and $Y$. Now let $X$ and $Y$ be normal operators.
Since $X$ and $Y$ are normal operators, the spectral theorem (see \cite{Mu}) implies that
$|X^*| = |X|$, $|Y^*| = |Y|$ and hence, by Theorem \ref{T.144} we obtain
\begin{align*}
{\|\,|X| - |Y|\,\|}_{_2} \leq {\|X - Y\|}_{_2}.
\end{align*}
Therefore, if $X$ and $Y$ are restricted to be normal,
then the best coefficient in the Araki--Yamagami inequality is $1$ instead of $\sqrt{2}$.
\end{remark}
The following result is a special case of \cite[Theorem~2.1]{Lee}.
\begin{theorem}\label{T.148}
If $X$ and $Y$ are operators on $\mathcal{H}$, then
\begin{align*}
{\|X + Y\|}^2_{_2} \leq {\big\|\,|X^*| + |Y^*|\,\big\|}_{_2}{\big\|\,|X| + |Y|\,\big\|}_{_2} .
\end{align*}
\end{theorem}
\begin{proof}
We may assume that $X, Y\neq0$ otherwise the desired inequality trivially holds.
By \eqref{I.2.1}, Theorem \ref{T.14}, Theorem \ref{T.18}(i) and the arithmetic-geometric mean inequality we have
\begingroup\makeatletter\def\f@size{10}\check@mathfonts
\begin{align*}
{\big\|\,|X^*| + |Y^*|\,\big\|}^2_{_2}{\big\|\,|X| + |Y|\,\big\|}^2_{_2}
& = \Big({\big\|\,|X^*|\,\big\|}^2_{_2} + {\big\|\,|Y^*|\,\big\|}^2_{_2} + 2{\big\|\,|X^*|\,\big\|}_{_2}{\big\|\,|Y^*|\,\big\|}_{_2}\cos\Theta_{_{|X^*|,|Y^*|}}\Big)
\\& \qquad \times \Big({\big\|\,|X|\,\big\|}^2_{_2} + {\big\|\,|Y|\,\big\|}^2_{_2} + 2{\big\|\,|X|\,\big\|}_{_2}{\big\|\,|Y|\,\big\|}_{_2}\cos\Theta_{_{|X|,|Y|}}\Big)
\\& = \Big({\|X\|}^2_{_2} + {\|Y\|}^2_{_2} + 2{\|X\|}_{_2}{\|X\|}_{_2}\cos\Theta_{_{|X^*|,|Y^*|}}\Big)
\\& \qquad \times \Big({\|X\|}^2_{_2} + {\|Y\|}^2_{_2} + 2{\|X\|}_{_2}{\|Y\|}_{_2}\cos\Theta_{_{|X|,|Y|}}\Big)
\\& = \Big({\|X\|}^2_{_2} + {\|Y\|}^2_{_2}\Big)^2 + 4{\|X\|}^2_{_2}{\|Y\|}^2_{_2}\cos\Theta_{_{|X^*|,|Y^*|}}\cos\Theta_{_{|X|,|Y|}}
\\& \qquad + 2{\|X\|}_{_2}{\|Y\|}_{_2}\Big({\|X\|}^2_{_2} + {\|Y\|}^2_{_2}\Big)\Big(\cos\Theta_{_{|X^*|,|Y^*|}}+\cos\Theta_{_{|X|,|Y|}}\Big)
\\& \geq \Big({\|X\|}^2_{_2} + {\|Y\|}^2_{_2}\Big)^2 + 4{\|X\|}^2_{_2}{\|Y\|}^2_{_2}\cos\Theta_{_{|X^*|,|Y^*|}}\cos\Theta_{_{|X|,|Y|}}
\\& \qquad + 4{\|X\|}_{_2}{\|Y\|}_{_2}\Big({\|X\|}^2_{_2} + {\|Y\|}^2_{_2}\Big)\sqrt{\cos\Theta_{_{|X^*|,|Y^*|}}\cos\Theta_{_{|X|,|Y|}}}
\\& \geq \Big({\|X\|}^2_{_2} + {\|Y\|}^2_{_2}\Big)^2 + 4{\|X\|}^2_{_2}{\|Y\|}^2_{_2}\cos^2\Theta_{_{X,Y}}
\\& \qquad + 4{\|X\|}_{_2}{\|Y\|}_{_2}\Big({\|X\|}^2_{_2} + {\|Y\|}^2_{_2}\Big)\big|\cos\Theta_{_{X,Y}}\big|
\\& \geq \Big({\|X\|}^2_{_2} + {\|Y\|}^2_{_2}\Big)^2 + 4{\|X\|}^2_{_2}{\|Y\|}^2_{_2}\cos^2\Theta_{_{X,Y}}
\\& \qquad + 4{\|X\|}_{_2}{\|Y\|}_{_2}\Big({\|X\|}^2_{_2} + {\|Y\|}^2_{_2}\Big)\cos\Theta_{_{X,Y}}
\\& = \Big({\|X\|}^2_{_2} + {\|Y\|}^2_{_2} + 2{\|X\|}_{_2}{\|Y\|}_{_2}\cos\Theta_{_{X,Y}} \Big)^2 = {\|X + Y\|}^4_{_2}.
\end{align*}
\endgroup
\end{proof}
Next, we provide alternative proof of an inequality for the Hilbert--Schmidt norm due to F.~Kittaneh \cite[Theorem~2.1]{Kit.5}.
\begin{theorem}\label{T.149}
If $X$ and $Y$ are operators on $\mathcal{H}$, then
\begin{align*}
{\big\|\,|X| - |Y|\,\big\|}^2_{_2} \leq {\|X + Y\|}_{_2}{\|X - Y\|}_{_2}.
\end{align*}
\end{theorem}
\begin{proof}
Since the desired inequality trivially holds when $X= 0$ or $Y=0$, we may assume $X, Y\neq0$.
By the arithmetic-geometric mean inequality we have
\begingroup\makeatletter\def\f@size{10}\check@mathfonts
\begin{align*}
{\|X\|}_{_2}{\|Y\|}_{_2}\big(1+\cos\Theta_{_{|X|,|Y|}}\big) \leq 2{\|X\|}_{_2}{\|Y\|}_{_2} \leq {\|X\|}^2_{_2} + {\|Y\|}^2_{_2}.
\end{align*}
\endgroup
Then, since $\cos\Theta_{_{|X|,|Y|}}\geq 0$, we obtain
\begingroup\makeatletter\def\f@size{10}\check@mathfonts
\begin{align*}
{\|X\|}_{_2}{\|Y\|}_{_2}\Big(\cos\Theta_{_{|X|,|Y|}} + \cos^2\Theta_{_{|X|,|Y|}}\Big)
\leq \cos\Theta_{_{|X|,|Y|}}\Big({\|X\|}^2_{_2} + {\|Y\|}^2_{_2}\Big),
\end{align*}
\endgroup
and so
\begingroup\makeatletter\def\f@size{10}\check@mathfonts
\begin{align}\label{T.I.1.149}
{\|X\|}_{_2}{\|Y\|}_{_2}\cos\Theta_{_{|X|,|Y|}}
\leq \cos\Theta_{_{|X|,|Y|}}\Big({\|X\|}^2_{_2} + {\|Y\|}^2_{_2}\Big) - {\|X\|}_{_2}{\|Y\|}_{_2}\cos^2\Theta_{_{|X|,|Y|}}.
\end{align}
\endgroup
Therefore by \eqref{I.2.1}, Theorem \ref{T.14}, Theorem \ref{T.18}(ii) and \eqref{T.I.1.149} we have
\begingroup\makeatletter\def\f@size{10}\check@mathfonts
\begin{align*}
{\|X + Y\|}^2_{_2}{\|X - Y\|}^2_{_2}
& = \Big({\|X\|}^2_{_2} + {\|Y\|}^2_{_2}\Big)^2 - 4{\|X\|}^2_{_2}{\|Y\|}^2_{_2}\cos^2\Theta_{_{X,Y}}
\\& \geq \Big({\|X\|}^2_{_2} + {\|Y\|}^2_{_2}\Big)^2 - 4{\|X\|}^2_{_2}{\|Y\|}^2_{_2}\cos\Theta_{_{|X|,|Y|}}
\\& \geq \Big({\|X\|}^2_{_2} + {\|Y\|}^2_{_2}\Big)^2 - 4{\|X\|}_{_2}{\|Y\|}_{_2}\Big({\|X\|}^2_{_2} + {\|Y\|}^2_{_2}\Big)\cos\Theta_{_{|X|,|Y|}}
\\& \qquad \qquad \qquad \qquad + 4{\|X\|}^2_{_2}{\|Y\|}^2_{_2}\cos^2\Theta_{_{|X|,|Y|}}
\\& = \Big({\|X\|}^2_{_2} + {\|Y\|}^2_{_2} - 2{\|X\|}_{_2}{\|Y\|}_{_2}\cos\Theta_{_{|X|,|Y|}}\Big)^2
\\& = \Big({\big\|\,|X|\,\big\|}^2_{_2} + {\big\|\,|Y|\,\big\|}^2_{_2} - 2{\big\|\,|X|\,\big\|}_{_2}{\big\|\,|Y|\,\big\|}_{_2}\cos\Theta_{_{|X|,|Y|}}\Big)^2
= {\big\|\,|X| - |Y|\,\big\|}^4_{_2}.
\end{align*}
\endgroup
\end{proof}
The following result may be stated as well.
\begin{theorem}\label{T.1406}
If $X$ and $Y$ are operators on $\mathcal{H}$, then
\begin{align*}
{\big\|\,|X^*| + |Y^*|\,\big\|}_{_2} \leq \sqrt{2}{\big\|\,|X| + |Y|\,\big\|}_{_2}.
\end{align*}
\end{theorem}
\begin{proof}
We may assume that $X, Y\neq0$ otherwise the desired inequality trivially holds.
Since
$-2 \leq \cos\Theta_{_{|X^*|,|Y^*|}} -2\cos\Theta_{_{|X|,|Y|}} \leq 1$,
by the arithmetic-geometric mean inequality we get
\begingroup\makeatletter\def\f@size{10}\check@mathfonts
\begin{align*}
2{\|X\|}_{_2}{\|Y\|}_{_2}\Big(\cos\Theta_{_{|X^*|,|Y^*|}} -2\cos\Theta_{_{|X|,|Y|}}\Big)
\leq 2{\|X\|}_{_2}{\|Y\|}_{_2} \leq {\|X\|}^2_{_2} + {\|Y\|}^2_{_2}.
\end{align*}
\endgroup
Hence
\begingroup\makeatletter\def\f@size{10}\check@mathfonts
\begin{align}\label{T.I.1.1406}
2{\|X\|}_{_2}{\|Y\|}_{_2}\cos\Theta_{_{|X^*|,|Y^*|}} \leq {\|X\|}^2_{_2} + {\|Y\|}^2_{_2}
+ 4{\|X\|}_{_2}{\|Y\|}_{_2}\cos\Theta_{_{|X|,|Y|}}.
\end{align}
\endgroup
So, by \eqref{I.2.1}, Theorem \ref{T.14} and \eqref{T.I.1.1406} we have
\begingroup\makeatletter\def\f@size{10}\check@mathfonts
\begin{align*}
{\big\|\,|X^*| + |Y^*|\,\big\|}^2_{_2}
& = {\big\|\,|X^*|\,\big\|}^2_{_2} + {\big\|\,|Y^*|\,\big\|}^2_{_2} + 2{\big\|\,|X^*|\,\big\|}_{_2}{\big\|\,|Y^*|\,\big\|}_{_2}\cos\Theta_{_{|X^*|,|Y^*|}}
\\& = {\|X\|}^2_{_2} + {\|Y\|}^2_{_2} + 2{\|X\|}_{_2}{\|Y\|}_{_2}\cos\Theta_{_{|X^*|,|Y^*|}}
\\& \leq 2{\|X\|}^2_{_2} + 2{\|Y\|}^2_{_2} + 4{\|X\|}_{_2}{\|Y\|}_{_2}\cos\Theta_{_{|X|,|Y|}}
\\& = 2\Big({\big\|\,|X|\,\big\|}^2_{_2} + {\big\|\,|Y|\,\big\|}^2_{_2} + 2{\big\|\,|X|\,\big\|}_{_2}{\big\|\,|Y|\,\big\|}_{_2}\cos\Theta_{_{|X|,|Y|}}\Big)
= 2{\big\|\,|X| + |Y|\,\big\|}^2_{_2}.
\end{align*}
\endgroup
\end{proof}
In \cite[p.~584]{Lee}, E.-Y.~Lee conjectured that for arbitrary $n$-by-$n$
matrices $A$ and $B$, the inequality
\begin{align*}
{\big\|A + B\big\|}_{_2} \leq \sqrt{\frac{\sqrt{2} + 1}{2}}\,{\big\|\,|A| + |B|\,\big\|}_{_2}
\end{align*}
holds. We end this section by a proof of Lee's conjecture for operators.
\begin{theorem}\label{T.14060}
If $X$ and $Y$ are operators on $\mathcal{H}$, then
\begin{align*}
{\big\|X + Y\big\|}_{_2} \leq \sqrt{\frac{\sqrt{2} + 1}{2}}\,{\big\|\,|X| + |Y|\,\big\|}_{_2}.
\end{align*}
\end{theorem}
\begin{proof}
Since the desired inequality trivially holds when $X= 0$ or $Y=0$, we may assume $X, Y\neq0$.
By \eqref{I.2.1}, Theorem \ref{T.14}, Theorem \ref{T.18}(ii) we have
\begingroup\makeatletter\def\f@size{10}\check@mathfonts
\begin{align*}
\frac{\sqrt{2} + 1}{2}\,{\big\|\,|X| + |Y|\,\big\|}^2_{_2}
& = \frac{\sqrt{2} + 1}{2}\Big({\big\|\,|X|\,\big\|}^2_{_2} + {\big\|\,|Y|\,\big\|}^2_{_2} + 2{\big\|\,|X|\,\big\|}_{_2}{\big\|\,|Y|\,\big\|}_{_2}\cos\Theta_{_{|X|,|Y|}}\Big)
\\& = \frac{\sqrt{2} + 1}{2}\Big({\|X\|}^2_{_2} + {\|Y\|}^2_{_2} + 2{\|X\|}_{_2}{\|Y\|}_{_2}\cos\Theta_{_{|X|,|Y|}}\Big)
\\& \geq \frac{\sqrt{2} + 1}{2}\Big({\|X\|}^2_{_2} + {\|Y\|}^2_{_2} + 2{\|X\|}_{_2}{\|Y\|}_{_2}\cos^2\Theta_{_{X,Y}}\Big)
\\& = {\|X\|}^2_{_2} + {\|Y\|}^2_{_2} + 2{\|X\|}_{_2}{\|Y\|}_{_2}\cos\Theta_{_{X,Y}}
+ \frac{\sqrt{2}-1}{2}\Big({\|X\|}^2_{_2} + {\|Y\|}^2_{_2}\Big)
\\& \qquad \qquad \qquad \qquad + {\|X\|}_{_2}{\|Y\|}_{_2}\Big((\sqrt{2} + 1)\cos^2\Theta_{_{X,Y}} - 2\cos\Theta_{_{X,Y}}\Big)
\\& = {\big\|X + Y\big\|}^2_{_2} +
\frac{\sqrt{2}-1}{2}\big({\|X\|}_{_2} - {\|Y\|}_{_2}\big)^2
\\& \qquad \qquad \qquad + {\|X\|}_{_2}{\|Y\|}_{_2}\Big((\sqrt{2} - 1) + (\sqrt{2} + 1)\cos^2\Theta_{_{X,Y}} - 2\cos\Theta_{_{X,Y}}\Big)
\\& = {\big\|X + Y\big\|}^2_{_2} +
\frac{\sqrt{2}-1}{2}\big({\|X\|}_{_2} - {\|Y\|}_{_2}\big)^2
\\& \qquad \qquad \qquad + {\|X\|}_{_2}{\|Y\|}_{_2}\Big(\sqrt{\sqrt{2} - 1} - \sqrt{\sqrt{2} + 1}\cos\Theta_{_{X,Y}}\Big)^2
\\& \geq {\big\|X + Y\big\|}^2_{_2}.
\end{align*}
\endgroup
\end{proof}
\begin{remark}\label{R.129}
Suppose $\mathcal{M}_2(\mathbb{C})$ is the algebra of all complex $2\times2$ matrices.
Let ${\rm Det}(A)$ denote the determinant of $A\in\mathcal{M}_2(\mathbb{C})$.
Recall (see e.g. \cite[p.~460]{Franca}) that for $A \in \mathcal{M}_2(\mathbb{C})$ we have
\begin{align*}
|A| = \frac{1}{\sqrt{{\rm Tr}(A^*A) + 2\sqrt{{\rm Det}(A^*A)}}}\Big(\sqrt{{\rm Det}(A^*A)}\,I + A^*A\Big).
\end{align*}
Now, let $X=\begin{bmatrix}
0 & 0 \\
-1 & 0
\end{bmatrix}$, $Y=\begin{bmatrix}
0 & 0 \\
0 & 1
\end{bmatrix}$ and $Z=\begin{bmatrix}
0 & 0 \\
1-\sqrt{2} & \sqrt{\sqrt{8}-2}
\end{bmatrix}$.
Then simple computations show that $|X|= \begin{bmatrix}
1 & 0 \\
0 & 0
\end{bmatrix}$, $|X^*|= \begin{bmatrix}
0 & 0 \\
0 & 1
\end{bmatrix}$, $|Y| = |Y^*|= \begin{bmatrix}
0 & 0 \\
0 & 1
\end{bmatrix}$ and
$|Z|=\begin{bmatrix}
3-\sqrt{8} & -\sqrt{\sqrt{200}-14} \\
-\sqrt{\sqrt{200}-14} & \sqrt{8}-2
\end{bmatrix}$.
Therefore,
\begin{align*}
{\big\|\,|X^*| + |Y^*|\,\big\|}_{_2}=
{\left\|\begin{bmatrix}
2 & 0 \\
0 & 0
\end{bmatrix}\right\|}_{_2} = \sqrt{{\rm Tr}\left(\begin{bmatrix}
4 & 0 \\
0 & 0
\end{bmatrix}\right)}=2
\end{align*}
and
\begin{align*}
\sqrt{2}{\big\|\,|X| + |Y|\,\big\|}_{_2} = \sqrt{2}{\left\|\begin{bmatrix}
1 & 0 \\
0 & 1
\end{bmatrix}\right\|}_{_2} = \sqrt{{\rm Tr}\left(\begin{bmatrix}
1 & 0 \\
0 & 1
\end{bmatrix}\right)} = 2.
\end{align*}
Hence the inequality in Theorem \ref{T.1406} is sharp.
In addition, we have
\begingroup\makeatletter\def\f@size{10}\check@mathfonts
\begin{align*}
{\big\|X + Z\big\|}_{_2} = {\left\|\begin{bmatrix}
0 & 0 \\
-\sqrt{2} & \sqrt{\sqrt{8}-2}
\end{bmatrix}\right\|}_{_2} = \sqrt{{\rm Tr}\left(\begin{bmatrix}
2 & -\sqrt{\sqrt{32}-4} \\
-\sqrt{\sqrt{32}-4} & \sqrt{8}-2
\end{bmatrix}\right)}=\sqrt[4]{8}
\end{align*}
\endgroup
and
\begingroup\makeatletter\def\f@size{10}\check@mathfonts
\begin{align*}
\sqrt{\frac{\sqrt{2} + 1}{2}}\,{\big\|\,|X| + |Z|\,\big\|}_{_2} &=
\sqrt{\frac{\sqrt{2} + 1}{2}}\,{\left\|\begin{bmatrix}
4-\sqrt{8} & -\sqrt{\sqrt{200}-14} \\
-\sqrt{\sqrt{200}-14} & \sqrt{8}-2
\end{bmatrix}\right\|}_{_2}
\\&= \sqrt{\frac{\sqrt{2} + 1}{2}}\,\sqrt{{\rm Tr}\left(\begin{bmatrix}
10-\sqrt{72} & -\sqrt{\sqrt{3200}-56} \\
-\sqrt{\sqrt{3200}-56} & \sqrt{8}-2
\end{bmatrix}\right)}
\\& = \sqrt{\frac{\sqrt{2} + 1}{2}}\,\sqrt{2\sqrt{8}-4} =\sqrt[4]{8}.
\end{align*}
\endgroup
So, the inequality in Theorem \ref{T.14060} is also sharp.
\end{remark}
\bibliographystyle{amsplain}

\end{document}